\documentclass[a4paper,12pt]{article}

\title{Polar coordinates view on KM-arcs}
\author{Kanat Abdukhalikov and Duy Ho\\	
UAE University, PO Box 15551, Al Ain, UAE\\
abdukhalik@uaeu.ac.ae, duyho92@gmail.com}

\date{ }

\usepackage{amsthm,amsmath,amssymb} 

\begin{document} 

\maketitle

\theoremstyle{plain} 
\newtheorem{theorem}{Theorem}
\newtheorem{corollary}{Corollary}
\newtheorem{lemma}{Lemma}
\newtheorem{observation}{Observation}
\newtheorem{proposition}{Proposition}

\theoremstyle{definition} 
\newtheorem{definition}{Definition}
\newtheorem{claim}{Claim}
\newtheorem{fact}[theorem]{Fact}
\newtheorem{assumption}{Assumption}
\newtheorem{remark}{Remark}
\newtheorem{example}{Example}


\begin{abstract}  
We study presentations of KM-arcs in polar coordinates.  New cha\-racterizations on the point set of KM-arcs are obtained in terms of power sums and bilinear forms.   Therefore, we provide purely algebraic criteria for the
existence of KM-arcs. We also describe some constructions of KM-arcs which include known examples and describe their automorphism groups in this presentation.	
\end{abstract} 


\section{Introduction}

In \cite{korchmaros1990}, Korchm\'aros and Mazzocca initiated the study of \textit{$(q+t, t)$-arc of type $(0, 2, t)$} in the projective plane $PG(2,q)$.  These objects are sets of $q +t$ points meeting every line in $0, 2$ or $t$ points. They are also called \textit{KM-arcs of type $t$} in honour of the authors of \cite{korchmaros1990}, as mentioned in \cite{geertrui2019}. KM-arcs were studied also in 
\cite{geertrui2015,gacs2003,vandendriessche2011,vandendriessche2017}.

KM-arcs have been investigated exclusively  in homogeneous coordinates. This presentation has its own advantages with the established algebraic methods over finite fields.  On the other hand, the drawback of this approach is that there are  points of KM-arcs presented on the infinite line. 
Descriptions for structural results of KM-arcs are then  required to be split between points in the affine plane and infinite points, as seen in the constructions by G\'acs and Weiner in \cite{gacs2003}. 

In this paper, we describe KM-arcs in polar coordinates. This approach was initiated by investigations in \cite{Ab2017,Ab2019,Ab2019b,Ab2021}. With this approach, KM-arcs can be described solely with points in the affine plane. As a consequence, structural results of KM-arcs can be obtained uniformly on the point set. To demonstrate this point of view, we provide new characterizations of  KM-arcs in terms of bilinear forms and power sums. 
We also describe some examples of KM-arcs in this presentation.

The paper is organized as follows. In Section 2, we recall preliminary results on projective planes, affine planes and KM-arcs. In Section 3, we describe a configuration in Definition \ref{starset} that we use to describe KM-arcs in polar coordinates. Also in this section we obtain characterizations of KM-arcs in terms of bilinear forms and power sums, in Theorem \ref{thmbracket} and Theorem \ref{Eset}, respectively. 
In Section 4 and Section 5,    we describe  constructions of KM-arcs in $PG(2,q)$ using hyperovals and KM-arcs in $PG(2,r)$, where $q$ and $r$ are powers of $2$ with $q=r^l$. 
We also show that the KM-arcs constructed in Section 5 are
translation KM-arcs and describe their automorphisms.
 

%

\section{Preliminaries}

In this section we recall some definitions and notation. 

\subsection{Polar representation}
In this paper we consider finite fields only in characteristics $2$. 
Let $F = \mathbb{F}_{2^m}$ be a finite field of order $q = 2^m$. Consider $F$ as a subfield of $K = \mathbb{F}_{2^n}$, where
$n = 2m$, so $K$ is a two dimensional vector space over $F$.

The \textit{conjugate} of $x \in K$ over $F$ is $$\bar{x} = x^q.$$
Then the \textit{trace} and the \textit{norm} maps from $K$ to $F$ are
$$T(x) = Tr_{K/F}(x) = x + \bar{x} = x + x^q,$$
$$N(x) = N_{K/F}(x) = x\bar{x} = x^{1+q}.$$

The \textit{unit circle} of $K$ is the set of elements of norm $1$:
$$S = \{u \in K \mid u \bar{u}= 1 \}.$$

Therefore, $S$ is the multiplicative group of $(q +1)$st roots of unity in $K$. 
Since $F \cap S = \{1\}$, each non-zero element of $K$ has a unique polar coordinate representation $x = \lambda u$
with $\lambda \in F^*$ and $u \in S$. For any $x \in K^*$ we have 
$\lambda = \sqrt{x \bar{x}}$ and $u = \sqrt{x /\bar{x}}$.

One can define a nondegenerate bilinear form $\langle \cdot, \cdot \rangle : K \times K \rightarrow F$ by
$$\langle x, y \rangle = T(x\bar{y}) = x\bar{y} + \bar{x}y.$$
Then the form  $\langle \cdot, \cdot \rangle$ is alternating and symmetric, that is, 
$\langle a, a \rangle=0$ and 
$\langle a, b \rangle= \langle b, a \rangle$.

Following \cite{fisher2006}, consider an element $\mathbf{i} \in K$ with property $T(\mathbf{i}) = \mathbf{i}+\mathbf{i}^q = 1$. Then $K = F(\mathbf{i})$ and $\mathbf{i}$ is a root of a quadratic equation
$$
z^2 + z + \delta = 0,
$$
where $\delta = N(\mathbf{i}) \in F$. Any element $z \in K$ can be represented as $z = x + y\mathbf{i}$, where $x, y \in F$. For $z = x + y\mathbf{i}$ we have $x = \langle \mathbf{i}, z\rangle$, and $y = \langle 1, z \rangle$.


\subsection{Affine and projective planes} 

Consider points of a projective plane $PG(2, q)$ in homogeneous coordinates \cite{Hir} as triples  $(x:y:z)$, where $x,y,z\in F$, $(x,y, z) \ne (0, 0, 0)$, and we identify $(x : y : z)$ with $(\lambda x : \lambda y : \lambda z)$, $\lambda \in F^*$. Then the points of $PG(2, q)$ are
$$
\{(x : y : 1) \mid x \in F, y \in F \} \cup \{ (x : 1 : 0) \mid x \in F\} \cup \{(1 : 0 : 0)\}.
$$
For $a, b, c \in F$, $(a, b, c) \ne (0, 0, 0)$, the line $[a : b : c]$ in $PG(2, q)$ is defined as
$$
[a : b : c] = \{(x : y : z) \in PG(2, q) \mid ax + by + cz = 0\}.
$$

Triples $[a : b : c]$ and $[\lambda a : \lambda b : \lambda c]$ with $\lambda \in F^*$ define the same lines. The point $(x : y : z)$ is incident with the line $[a : b : c]$ if and only if $ax + by + cz = 0$. We shall call points of the form $(x : y : 0)$ the points at infinity. Then $[0 : 0 : 1]$ indicates the line at infinity.

We define an affine plane $AG(2, q) = PG(2, q) \backslash [0 : 0 : 1]$, so points of this affine plane $AG(2, q)$ are $\{(x : y : 1) \mid  x, y \in F \}$.
Associating $(x : y : 1)$ with $(x, y)$ we can identify points of the affine plane $AG(2, q)$ with elements of the vector space $V(2, q) = \{(x, y) \mid x, y \in F\}$, and we will write $AG(2, q) = V (2, q)$.
Lines in $AG(2, q) = V (2, q)$ are $\{(c, y) \mid y \in F\}$ and $\{(x, xb + a) \mid x \in F\}, a, b, c \in F$.
These lines can be described by equations $x = c$ and $y = xb + a$.

We introduce now another representation of $PG(2, q)$ using the field $K$. Consider pairs $(x : z)$, where $x \in K, z \in F$, $x \ne 0$ or $z \ne 0$, and we identify $(x : z)$ with $(\lambda x : \lambda z),  \lambda \in F^*$. 
Then points of $PG(2, q)$ are
$$
\{(x : 1) \mid x \in K\} \cup \{(u : 0) \mid u \in S \}.
$$
For $\alpha\in K$, $\beta\in F$, $(\alpha,\beta) \ne (0,0)$,  we define the lines $[\alpha:\beta]$ in $PG(2, q)$ as
$$
[\alpha:\beta] = \{(x:z) \in PG(2,q) \mid \langle\alpha,x\rangle+\beta z = 0\}.
$$
Pairs $[\alpha:\beta]$ and $[\lambda\alpha:\lambda\beta]$ with $\lambda \in F^*$ define the same lines. The point $(x : z)$ is incident with the line $[\alpha:\beta]$ if and only if $\langle\alpha,x\rangle+\beta z = 0$. The element $u_\infty = (u : 0), u \in S$, will be referred to as the point at infinity in the direction of $u$. So $[0 : 1]$ indicates the line at infinity.

We define an affine plane $AG(2, q) = PG(2, q) \backslash [0 : 1]$, so points of  this affine plane $AG(2,q)$ are $\{(x : 1) \mid x \in K\}$.
Associating $(x : 1)$ with $x \in K$ we can identify points of the affine plane $AG(2,q)$ with elements of the field $K$, and we write $AG(2, q) = K.$ 
Lines of $AG(2, q)=K$ are of the form
$$
L(u,\mu) = \{x \in K \mid \langle u, x\rangle + \mu = 0\},
$$
where $u \in S$ and $\mu \in F$ (cp. \cite[subsection 2.1]{ball1999}).

Throughout the paper, we will consider these two representations of the projective plane $PG(2, q)$, and for each of such projective planes we consider a fixed affine plane $AG(2, q)$ described above. 
They will be written as $AG(2, q) = V (2, q)$ and $AG(2, q) = K$.

\subsection{KM-arcs and basic facts}
 
In the projective plane $PG(2,q)$, a \textit{KM-arc of type $t$} (also known as a \textit{$(q+t, t)$-arc of type $(0, 2, t)$}) is a set $H$ of $q +t$ points meeting every line in $0, 2$ or $t$ points. 

If $H$ is a KM-arc of type $t$ in $PG(2, q)$, $2 < t < q$, then
\begin{enumerate}
	\item $q$ is even and $t$ is a divisor of $q$;
	\item each point of $H$ is on exactly one $t$-secant and every other line through this point is a 2-secant of $H$, cp.  \cite{korchmaros1990};
	\item  there are $q/t + 1$ different $t$-secants to $H$, and they are concurrent at a unique point called the \textit{t-nucleus} of $H$, cp. \cite{gacs2003};
	\item all other lines contain 0 or 2 points of $H$, cp. \cite{vandendriessche2017}.
\end{enumerate}

 The case $t = q$ is classified as $H$ can be shown to be
 symmetric difference of two lines.  
 The case $t=1$ is a degenerate case when $H$ is an \textit{oval} of $q+1$ points and therefore has a unique nucleus, which we call the $1$-nucleus of $H$. When $t=2$, the set $H$ is called a \textit{hyperoval}.   In this case, if $p$ is a point in $PG(2,q)$ not on $H$, then there are $q/2 + 1$ different $2$-secants of $H$ going through $p$. Thus, any point not on $H$ can be considered as a $2$-nucleus of $H$.  

\subsection{Vandermonde sets}

We recall sets with the following property first considered by  G\'{a}cs and Weiner \cite{gacs2003} (and subsequently by other authors in  \cite{blokhuis2017,sziklai2008}). Let $1 < t < q$. A set $T = \{y_1, \cdots, y_t\} \subseteq GF(q)$
is called a  \textit{Vandermonde set}  if
	\[
\pi_{k}(T) := \sum_{y \in T} y^k	=  0,
\]
 for all $1\le k \le t-2$.  
 
\begin{lemma}[cp.  \cite{sziklai2008}, \cite{Ab2019c}] \label{vanaffine} For a set of $t$ elements the Vandermonde 
property is invariant under transformations $y \rightarrow ay+b, a\ne0$,  if and only if $t$ is even or $b = 0$.
\end{lemma} 

Some examples of Vandermonde sets can be found in \cite[Proposition 1.8]{sziklai2008}. In Section 4, we will use the fact that any additive subgroup of $GF(q)$ is a Vandermonde set.

\section{KM-arcs in polar presentation} 

In this section we obtain characterizations of KM-arcs in terms of bilinear forms and power sums. 
 
\subsection{The configuration of a KM-arc in polar presentation}
\begin{lemma} Any KM-arc in $PG(2,q)$ is equivalent to a KM-arc in $AG(2,q)$ with $t$-nucleus at $0$.
\end{lemma}
\begin{proof} Assume $H$ has a point on infinity line $L_\infty$. Let $L$  be a line which does not intersect $H$.  Consider a collineation $\gamma$ interchanging the lines $L$ and $L_\infty$. Then $\gamma(H)$ has no points at infinity and a suitable affine translation of $\gamma(H)$ is a KM-arc with $t$-nucleus at $0$.  
\end{proof}

\begin{lemma} \label{lmkmeven} For $t \ge 2$, let  $H$ be a set of $q+t$ points in $K=AG(2,q)$.  Assume that for every point $p$ of $H$ there exists a line $L$ going through $p$ such that $|L \cap H| \ge t$.
Then $H$ is a KM-arc of type $t$ if and only if every line intersects $H$ at an even number of points. 
\end{lemma}  
\begin{proof} If $H$ is a KM-arc of type $t$ then every line intersects $H$ at $0,2$ or $t$ points. 
		
	Assume every line intersects $H$ at an even number of points. Let $p$ be a point on $H$ and let 
		$L$ be a line going through $p$ such that $|L \cap H| \ge t$.
		Let $\mathcal{L}$ be the set of  $q+1$ lines going through $p$.   
		Since there are at most $q$ points in $H \backslash (H \cap L)$, each line in $\mathcal{L} \backslash L$ contains exactly two points of $H$ (including $p$). This also implies  $|L \cap H| =t$. In particular, every line intersects $H$ at $0,2$ or $t$ points and so $H$ is a KM-arc of type $t$. 
\end{proof}

 We now describe a configuration $H$ of points that is used throughout the paper. 

\begin{definition} \label{starset} Let $t \ge 2$. A set $H$ of $q+t$ points in $K=AG(2,q)$ is called a \textit{star-set} if all points of $H$ belong to a union of $\dfrac{q}{t}+1$ lines concurrent at $0$, and each line contains $t$ points of $H$. 
\end{definition}

%
%
%
%

By Lemma \ref{lmkmeven}, a star-set $H$ is a KM-arc of type $t \ge 2$ (with $t$-nucleus at $0$) if and only if every line intersects $H$ at an even number of points.

\begin{theorem}  \label{thmbracket}  Let $H$ be a star-set.  Then   $H$ is a KM-arc of type $t \ge 2$ with $t$-nucleus at $0$ if and only if for each $v \in S$ and $1 \le k \le q-2$,
	\begin{equation}  \label{eqbracket}
	\sum_{y \in H} \left\langle v, y\right\rangle^k=0.
	\end{equation} 
\end{theorem}
	\begin{proof}  The case $t=2$ is proved in \cite{Ab2019c}. In the remainder of the proof, we assume $t > 2$. Assume $H$ is a KM-arc of type $t$ with $t$-nucleus at $0$. Fix $v \in S$. For each $\mu \in F$, the line 
		$
		\left\langle v, x \right\rangle = \mu
		$
		intersects $H$ at an even number of points. In particular, the equation
		$$
		\left\langle v, x\right\rangle = \mu
		$$
		has an even number of solutions in $H$. This implies condition \eqref{eqbracket} is true for each $v \in S$ and $1 \le k \le q-2$.

		Conversely, assume condition \eqref{eqbracket} is true for each $v \in S$ and $1 \le k \le q-2$.  We will show that every line $\langle v, x\rangle = \mu$ intersects $H$ at an even number of points. Fix $v \in S$.   For $\mu=0$, the intersection $H_v$ of $H$ and the line  $\langle v, x\rangle =0$ is of size $0$ or $t$.    For each $\mu \in F^*$, let 
			$$Y_\mu := \left\{ y \in H  \mid \left\langle v, y\right\rangle = \mu \right\}.$$ 
			
			We note that, for each $1 \le k \le q-2$,  
		\[
		\sum_{y \in Y_\mu} \left\langle v, y\right\rangle^k= 
		\begin{cases}
		\mu^k & \text{if $|Y_\mu|$ is odd, }\\
		0         & \text{if $|Y_\mu|$ is even.}
		\end{cases} 
		\]
		
	Let $\Omega:= \{ \mu \in F^* \mid   |Y_\mu| \text{ is odd} \}$.  Assume that $\Omega \ne \varnothing$. Then
$$
0=\sum_{y \in H} \left\langle v, y \right\rangle^k
=\sum_{\mu \in \Omega}	\sum_{y \in Y_\mu} \left\langle v, y\right\rangle^k 
= \sum_{\mu \in \Omega} \mu^k.
$$
		Also, the sets $Y_\mu$ partition the set $ H \backslash H_v$ of even size (either $q$ or $q+t$),  so  that $|\Omega|$ is even. In particular, the sum of the vectors $(1, \mu, \mu^2, \cdots, \mu^{q-2})$, $\mu \in \Omega$, is the zero vector. 
		On the other hand, the Vandermonde's determinant implies that these vectors are linearly independent. Hence $\Omega = \varnothing$. 
			
		We have shown that  every line $\langle v, x\rangle = \mu$ intersects $H$ at an even number of points. By Lemma \ref{lmkmeven}, $H$ is a KM-arc of type $t$ with $t$-nucleus at $0$.   
	\end{proof}

\subsection{KM-arcs and power sums}
We recall and extend \cite[Lemma 9]{Ab2019c} to the following.
\begin{lemma} \label{bracketpower} The power of the bilinear form $\langle \cdot, \cdot \rangle$ is given by 
	$$
	\langle a,b \rangle^k = 
	\sum_{i=0}^{\lfloor(k-1)/2\rfloor} \binom{k}{i} \langle a^{iq+k-i},b^{iq+k-i} \rangle= 
	\sum_{i=\lceil(k+1)/2\rceil}^{k} \binom{k}{i} \langle a^{iq+k-i},b^{iq+k-i} \rangle.
	$$	  
\end{lemma}
\begin{proof} It was shown in \cite[Lemma 9]{Ab2019c} that
		$$
	\langle a,b \rangle^k = 
	\sum_{i=0}^{\lfloor(k-1)/2\rfloor} \binom{k}{i} \langle a^{iq+k-i},b^{iq+k-i} \rangle.
	$$
		For the remaining equality, when $k$ is odd we have
		\begin{align*}
		\langle a,b \rangle^k &= (a^qb+ab^q)^k= \sum_{i=0}^k \binom{k}{i} a^{iq}b^{i}a^{k-i}b^{(k-i)q}= \sum_{i=0}^k \binom{k}{i} a^{iq+k-i}b^{i+(k-i)q} \\
		&= \sum_{i=0}^{(k-1)/2} \binom{k}{i} a^{iq+k-i}b^{i+(k-i)q}
		+ \sum_{i={(k+1)/2}}^{k} \binom{k}{i}  a^{iq+k-i}b^{i+(k-i)q} \\
		&= \sum_{i={(k+1)/2}}^{k} \binom{k}{i}a^{i+(k-i)q}b^{iq+k-i}
		+  \sum_{i={(k+1)/2}}^{k} \binom{k}{i}  a^{iq+k-i}b^{i+(k-i)q} \\
		&= \sum_{i={(k+1)/2}}^{k}\binom{k}{i} \left(  a^{iq+k-i}b^{i+(k-i)q}+ a^{i+(k-i)q}b^{iq+k-i} \right) \\
		&=\sum_{i={(k+1)/2}}^{k} \binom{k}{i}
		\langle a^{iq+k-i},b^{iq+k-i} \rangle.
		\end{align*}   
		Similar to the above, when $k$ is even, we have
		\begin{align*}
		\langle a,b \rangle^k  &= \sum_{i=k/2+1}^{k} \binom{k}{i}\langle a^{iq+k-i},b^{iq+k-i} \rangle + \binom{k}{k/2} \langle a^{iq+k-i},b^{iq+k-i} \rangle \\
		&= \sum_{i=k/2+1}^{k} \binom{k}{i}\langle a^{iq+k-i},b^{iq+k-i} \rangle,
		\end{align*}   
		since $\binom{k}{k/2}$ is even.
\end{proof}

We recall the following partial ordering $\preceq$ on the set of nonnegative integers. If
$$
x=\sum_{i=0}^{m}x_i2^i \text{ and } y=\sum_{i=0}^{m}y_i2^i
$$
(where each $x_i$ and each $y_i$ is either $0$ or $1$), then $x \preceq y$ if and only if $x_i \le y_i$ for all $i$.
In other words, $x \preceq y$ if and only if all nonzero terms appearing in the
binary expansion of $x$ also appear in the binary expansion of $y$.

Let
\begin{equation*}  
D:= \{iq+k-i\mid 1 \le k \le q-2, 0\le i \le \lfloor(k-1)/2\rfloor, 
i \preceq k \}. 
\end{equation*}

\begin{theorem}  \label{Dset} 
	Let $H$ be a star-set.  Then $H$ is a KM-arc of type $t$  with $t$-nucleus at $0$  if and only if 
	\[
	\pi_{d}(H) := \sum_{y \in H} y^d	=  0,
	\]
	for all $d \in D$. 
\end{theorem} 	
\begin{proof}  Let $H= \{ y_1, \cdots y_{q+t}\}$. Fix $1\le k\le q-2$. 	For each $v \in S$, we have
			\begin{align*}
			\sum_{j=1}^{q+t} \left\langle v, y_j\right\rangle ^k 
			&= \sum_{j=1}^{q+t} \sum_{i=0}^{\lfloor(k-1)/2\rfloor}  \binom{k}{i} \left\langle v^{iq+k-i},y_j^{iq+k-i} \right\rangle\\
			&= \sum_{i=0}^{\lfloor(k-1)/2\rfloor}\sum_{j=1}^{q+t} \binom{k}{i} \left\langle v^{iq+k-i},y_j^{iq+k-i}\right\rangle \\ 
			&= \sum_{i=0}^{\lfloor(k-1)/2\rfloor}  \binom{k}{i} \left\langle v^{iq+k-i},\pi_{iq+k-i}(H)\right\rangle \\
			&= \sum_{\substack{ i \preceq k \\0 \le i \le \lfloor(k-1)/2\rfloor }} \left\langle v^{iq+k-i},\pi_{iq+k-i}(H)\right\rangle = P_k(v),
			\end{align*}
where $P_k$ is the polynomial obtained from expanding the terms in the last sum. We can assume $P_k$ has degree at most $q$, as $v^{q+1}=1$.
			
			The power of $v$ in $P_k(v)$ is either of the form $k-2i$ or $q+1+2i-k$. If $i \ne i'$, then $k-2i \ne k-2i'$, and  $k-2i \ne 2i'-k+q+1$.	Hence $P_k(v)$ has the form $P_k(v) = \sum \pi_lv^s$. 
			
			By Theorem \ref{thmbracket}, $H$  is a KM-arc of type $t$  with $t$-nucleus at $0$ if and only if	$P_k(v)=0$ has $q+1$ roots $v \in S$ for each $1 \le k \le q-2$. Equivalently, the coefficients of $P_k$ are zeros, that is, $\pi_d(H)=0$ for all $d \in D$.    
\end{proof}

%

Now we introduce an analog of the set $D$, which contains more elements but looks more symmetric. 
Let 
\[
D':=\{k+(q-1)i \mid 1 \le k \le q-1, i \preceq k \},
\]
\[
E:= \left\{ \sum_{j=0}^{m-1} 2^j x_j >0 \mid x_j \in \{0,1,q\}\right\}.
\]
An element $x \in E$ can also be described in binary form as  
\[
x=\sum_{j=0}^{2m-1} 2^j x_j>0,
\]
where $x_j \in \{0,1\}$ for $0 \le j \le 2m-1$, and $(x_j, x_{m+j}) \ne (1,1)$, for $0 \le j \le  m-1$.

\begin{lemma} \label{D=E} $D'=E$. 
\end{lemma}
	\begin{proof} We rewrite
		$
		k= \sum_{j=0}^{m-1} 2^j k_j,
		$
		where $k_j \in \{0,1\}$. 	For $i \preceq k$, we rewrite
		\[
		i= \sum_{j=0}^{m-1} 2^j i_j,
		\]
		where $i_j \le k_j$. 
		Let $d:=k+(q-1)i\in D'$. Then
		\[
		d=\sum_{j=0}^{m-1} 2^j (k_j+(q-1)i_j).
		\]
		For each $j$, there are three possible cases for $k_j+(q-1)i_j$, as follows.
		\begin{enumerate}
			\item $k_j=0,i_j=0$. Then $k_j+(q-1)i_j= 0$.
			\item $k_j=1, i_j=0$.  Then $k_j+(q-1)i_j= 1$.
			\item $k_j=1, i_j=1$.  Then $k_j+(q-1)i_j= q$.
		\end{enumerate}
		This implies $d \in E$ so that $D' \subseteq E$. Conversely, for each $j$, each element of $\{0,1,q\}$ can be rewritten in the form of $k_j+(q-1)i_j$ as above,  so that $E \subseteq D'$ and consequently $D'=E$. 
	\end{proof}

\begin{theorem} \label{Eset} Let $H$ be a star-set. Then $H$ is a KM-arc of type $t$  with $t$-nucleus at $0$  if and only if $\pi_{e}(H) = 0$ for all $e \in E$.
\end{theorem} 
	\begin{proof} We note that $D \subset D'=E$. By Theorem \ref{Dset}, if $\pi_{e}(H) = 0$ for all $e \in E$, then $H$ is a $KM$-arc of type $t$ with $t$-nucleus at $0$. 
		
	Conversely, assume $H$ is a KM-arc of type $t$ with $t$-nucleus at $0$.  By Theorem \ref{thmbracket}, for each $v \in S$ and $1 \le k \le q-2$,
		\begin{equation}    \label{eqnbracket}
		\sum_{y \in H} \left\langle v, y\right\rangle^k=0.
		\end{equation} 
		For $k=q-1$, it can be checked that condition \eqref{eqnbracket} also holds. Using Lemma \ref{bracketpower}, for $1 \le k \le q-1$, 
		
\begin{eqnarray*}
\sum_{y \in H} \left\langle v, y\right\rangle^k
&=& \sum_{\substack{ i \preceq k \\0 \le i \le \lfloor(k-1)/2\rfloor }} \left\langle v^{iq+k-i},\pi_{iq+k-i}(H)\right\rangle \\
&=& \sum_{\substack{ i \preceq k \\ \lceil(k+1)/2\rceil \le i \le k}} \left\langle v^{iq+k-i},\pi_{iq+k-i}(H)\right\rangle=0.
\end{eqnarray*} 
Similar to the proof of Theorem \ref{Dset}, it follows that $\pi_{d}(H)=0$ for all $d \in D'$. 
By Lemma \ref{D=E}, $\pi_{e}(H)=0$ for all $e \in E$. 
	\end{proof} 

In homogeneous coordinates,	
KM-arcs are considered  with the line at infinity  $L_\infty$ as a  $t$-secant. Under this setting,  G\'{a}cs and Weiner \cite{gacs2003} proved that the set of points on $L_\infty$ of a KM-arc is  a Vandermonde set.  In the following, we extend this result in polar coordinates. 

\begin{proposition} \label{tsecantvan} For $t \ge 2$, let  $\{y_i \mid 1 \le i \le t\} \subset K$ be the set of  points on a $t$-secant of a KM-arc $H$ with $t$-nucleus at $0$. Then the set
	$$
	\{ 1/y_i \mid 1 \le i \le t\}
	$$
	is a Vandermonde set.
\end{proposition} 

\begin{proof}  We identify each point $z=z_1+z_2\mathbf{i} \in K$ as $(z_1:z_2:1)$ in homogeneous coordinates.  		  We note that the Vandermonde property is invariant under affine transformations by Lemma \ref{vanaffine}. 
	Under a suitable map $x\mapsto \alpha x, \alpha \in K$, we can assume that the $x$-axis is a $t$-secant of $H$. This means we can choose the set $\{y_i \mid 1 \le i \le t\}$ on the $x$-axis, so that  $y_i\in F$ and each $y_i$ is identified as $(y_i:0:1)$.
	
	Assume the line $x=\mu y$ is a $t$-secant of $H$. Let $\varphi$ be the collineation defined by
	\[
	\varphi((x:y:z))= (x+\mu y:z:y). 
	\]
	Under $\varphi$, each point $(y_i:0:1)$ is mapped to $(1:1/y_i:0)$ and the $t$-nucleus $(0:0:1)$  is mapped to $(0:1:0)$. 
	Also,  the $t$-secant  $x=\mu y$  is mapped to the $y$-axis. 
	By \cite[Proposition 2.4]{gacs2003}, the set 	$
	\{ 1/y_1, \cdots, 1/y_t\}
	$
	is a Vandermonde set. 
\end{proof}

\section{KM-arcs from hyperovals and KM-arcs} 

In this section we construct KM-arcs starting  from KM-arcs in smaller dimensions. 
Let $r=2^h$, where $h \mid m$.   Let $F'=\mathbb{F}_{r}$, and $K'=\mathbb{F}_{r^2}$. We denote the group of $(r+1)$-st roots of unity in $K'$ by $S'$. 

We  recall the relative trace map $Tr_{F/F'}: F \rightarrow F'$, where
$$
Tr_{F/F'}(x)=x^{q/r}+x^{q/r^2}+ \cdots +x^r+x.
$$
Let $V_1:= \{ x \in F \mid Tr_{F/F'}(x)=1 \}$.  The main result of this section is the following.

\begin{theorem} \label{construction4} 
	Assume $m/h$ is odd.  
	Let $H' \subset K' $ be a $(r+s,s)$-arc of type $s \ge 1$ with $s$-nucleus at $0$.  Let $V_c:= cV_1$ for some $c \in K^*$.  Then $H:= \{\lambda u \mid 1/\lambda \in V_c,u \in H'\}$ is a $(q+sq/r,sq/r)$-arc in $K$ of type $t=sq/r$ with $t$-nucleus at $0$. 
\end{theorem}
In particular, 

1) if $H'$ is an oval with $1$-nucleus at $0$, then $H$ is a $(q+q/r,q/r)$-arc of type $t=q/r$ with $t$-nucleus at $0$.

2) if $H'$ is a hyperoval not containing $0$, then $H$ is a $(q+2q/r,2q/r)$-arc of type $t=2q/r$ with $t$-nucleus at $0$.

The proof of Theorem \ref{construction4} is presented at the end of this section. 
In Lemmas \ref{Usetd} and \ref{genTset} we examine special power sums of the sets $H'$ and $V_1$. 

 Let
\[
\mathcal{A}:= \left\{ \dfrac{q}{r^j} \mid 1 \le j \le m/h \right\} \cup \{0\} = \left\{1,r,r^2, \cdots, \dfrac{q}{r} \right\} \cup \{ 0\}. 
\]

\begin{lemma} \label{Usetd} 
Let $k=\sum_{j=0}^{h-1}2^ja_j >0,$ where $a_j \in \mathcal{A}$. Let $d=k+(q-1)i$, for some $i \preceq k$.
	Then $\pi_d(H')=0$. 
\end{lemma}	
	\begin{proof} 
		Let 
		\[
		D'_r:=\{l+(r-1)i \mid 1 \le l \le r-1, i \preceq l \}.
		\]
		As in Lemma \ref{D=E}, $D'_r$ can also be expressed as 
		\[
		E_r := \left\{ \sum_{j=0}^{h-1} 2^j x_j >0\mid x_j \in \{0,1,r\}\right\}.
		\]  
		 Let $i \preceq k$. We rewrite $i=	\sum_{j=0}^{h-1} 2^ji_j,$ where $i_j \in \{0,a_j\}$.
			Then 
			\[
			d=k+(q-1)i = \sum_{j=0}^{h-1} 2^j\left(a_j+(q-1)i_j\right).
			\] 
			For each $j$, we note that $a_j+(q-1)i_j \in \{a_j,qa_j\}$. Since $q \pmod{r^2-1} \in \{1,r\}$ and $a_j \pmod{r^2-1} \in \{0,1,r\}$, it follows that 
			\[
			a_j+(q-1)i_j \pmod{r^2-1} \in \{0,1,r\}. 
			\]
			Then, for $u \in K'$, there exists $d' \in D'_r=E_r$ such that
			\[
			u^d=u^{k+(q-1)i}=u^{d'}. 
			\]
			From Theorem \ref{Eset},
			\[
			\pi_d(H')= \pi_{d'}(H')=0. 
			\] 
	\end{proof}

\begin{lemma}  \label{genTset}
Let  $1 \le  k   \le q-2$. 
	If $\pi_k(V_1) \ne 0$, then 
	\[
	k= q-1 - \sum_{j=0}^{h-1} 2^ja_j,
	\]
	where $a_j \in \mathcal{A}$ for each $0 \le j\le h-1$. 
	In this case, $\pi_k(V_1) = 1$.
	\end{lemma}
	\begin{proof}  Rewrite $k$ in the form
		\[
		k=\dfrac{q}{r}l+s,
		\]
		where $ 0 \le l \le r-1$,
		and $0 \le s \le q/r-1$.  We consider three cases depending on $l$. 
		\begin{enumerate}
			\item $l=0$. 
			Then $1 \le k \le q/r-1$. Let $V_0:= \{ x \in F \mid Tr_{F/\mathbb{F}_r}(x)=0 \}$, which is an additive subgroup of $F$ and hence a Vandermonde set of size $q/r$.  Since $V_1$ is a translation of $V_0$, by Lemma \ref{vanaffine},  it is also a Vandermonde set of size $q/r$.
			
			It follows that $\pi_k(V_1) \ne 0$ if and only if 
			\[
			k= \dfrac{q}{r}-1 = q-1-\sum_{j=0}^{h-1}2^j \dfrac{q}{r}.
			\]
			 For $x \in V_0\backslash \{0\}$, 
			\[
			x^{q/r-1}= (x^{q/r^2}+ \cdots +x^r+x) x^{-1}=x^{q/r^2-1}+ \cdots +x^{r-1}+1,
			\]
			and so $\pi_{q/r-1}(V_0)=1$. Note that the set $V_0\cup V_1$ is a Vandermonde set of size $2q/r$, so that $\pi_{q/r-1}(V_0\cup V_1)=0$. Therefore, $\pi_{q/r-1}(V_1)=\pi_{q/r-1}(V_0)=1$. 
			\item $l=1$. Then $k=q/r+s$, and for $x \in V_1$, we have
			\[
			x^k = x^{q/r}x^s=\left( x^{q/r^2}+ \cdots +x+1 \right)x^s. 
			\]
			To evaluate the sum $\sum\limits_{x \in V_1} x^k$, we consider three cases of each sum $\sum\limits_{x \in V_1} x^{q/r^i+s}$, depending on $2 \le i \le m/h$.
			\begin{enumerate}
				\item $q/r^i+s< q/r-1.$ Then
				$
				\sum\limits_{x \in V_1} x^{q/r^{i}+s}=0.
				$ 
				\item $q/r^i+s= q/r-1.$ Then
				$
				\sum\limits_{x \in V_1} x^{q/r^{i}+s}=1.
				$ 
				\item   $q/r^i+s>q/r-1$. Rewrite 
				$
				\dfrac{q}{r^i}+s=\dfrac{q}{r}+s_0,
				$
				where
				$
				0 \le s_0 \le \dfrac{q}{r^i}-1 .
				$
				Then
				$$
				x^{q/r^{i}+s} = x^{q/r+s_0}=\left(x^{q/r^2}+ \cdots +x+1\right)  x^{s_0}, 
				$$
				and since for each $2 \le i' \le m/h$,   
				$$
				\dfrac{q}{r^{i'}}+s_0 \le \dfrac{q}{r^{i'}}+\dfrac{q}{r^i} -1<\dfrac{q}{r}-1,
				$$
				it follows that 
				$$
				\sum_{x \in V_1} x^{q/r^{i}+s}=0.
				$$
				
			\end{enumerate}
			Therefore $\pi_k(V_1) \ne 0$ if and only if $\pi_k(V_1)=1$.  This occurs when
			\[
			s = \dfrac{q}{r}-\dfrac{q}{r^i}-1,
			\]
			if and only if
			\[
			k = 2\dfrac{q}{r}- \dfrac{q}{r^i}-1 = q-1- \left(\sum_{j=1}^{h-1}2^j\dfrac{q}{r}+\dfrac{q}{r^i}\right). 
			\]
			\item $l>1$. For $x \in V_1$, we have
			\begin{align*}
			x^k &= x^{lq/r}x^s = \left(x^{q/r^2}+\cdots x+1\right)^lx^s. 
			\end{align*}
			Let $x^c$ be a term of the expansion of the above. Rewrite $l$ as the binary expansion 
			$
			l= \sum\limits_{j=0}^{h-1}2^jl_j,
			$
			where $l_j\in\{0,1\}$, for $0 \le j \le h-1$. Then
			\[
			c=\sum_{j=0}^{h-1}2^jl_ja_j+s,
			\]
			where $a_j \in \mathcal{A} \backslash \left\{ \dfrac{q}{r} \right\}.$
			From the conditions on $l$ and $s$, we have
			\[
			c \le	\dfrac{q}{r^2}l+s \le \dfrac{q}{r^2}\left(r-1\right)+\dfrac{q}{r}-1=   2\dfrac{q}{r}-\dfrac{q}{r^2}-1.
			\]
			Furthermore, $c \ne 2\dfrac{q}{r}-\dfrac{q}{r^2}-1$, otherwise $l=r-1, s=\dfrac{q}{r}-1$ and $k=q-1$, which is excluded from our assumption.   From part 1) and part 2),  $\pi_c(V_1) \ne 0$  if and only if $\pi_c(V_1)=1$. This occurs when $c=q/r-1$,  if and only if
			\[
			s=\dfrac{q}{r}-1-\sum_{j=0}^{h-1}2^jl_ja_j,
			\]
			if and only if
			%
			\begin{align*}
			k &= \dfrac{q}{r}(l+1)-1-\sum_{j=0}^{h-1}2^jl_ja_j\\
			&= \dfrac{q}{r}-1+ \sum_{j=0}^{h-1}2^jl_j\dfrac{q}{r}-\sum_{j=0}^{h-1}2^jl_ja_j \\
			&= \dfrac{q}{r}-1+\sum_{j=0}^{h-1}2^j\dfrac{q}{r}     -\sum_{j=0}^{h-1}2^j(1-l_j)\dfrac{q}{r}-\sum_{j=0}^{h-1}2^jl_ja_j \\
			&=q-1     -\sum_{j=0}^{h-1}2^j\left[(1-l_j)\dfrac{q}{r}+l_ja_j\right]\\
			&= q-1 - \sum_{j=0}^{h-1} 2^ja'_j,
			\end{align*}
			where $a'_j \in \mathcal{A}$.  
		\end{enumerate}
		The proof now follows. 
	\end{proof}

	\begin{proof}[Proof of Theorem \ref{construction4}]
		 We will assume that $c=1$ since multiplication by $c \in K^*$ is a collineation of $PG(2,q)$. 	
		 \begin{enumerate}
			\item   
			Let $x,y$  be two points  from $H'$ such that the lines $xF'$ and $yF'$ are distinct in $K'$.   This implies $x/y \not \in F'$. But since $m/h$ is odd,  we have $F \cap K' = F'$.  Hence  $x/y \not \in F$, and so the lines $xF$ and $yF$ are distinct in $K$. 
			
			Assume $s \ge 2$. For $x,y \in F', x \ne y$, we consider two distinct points $x\alpha,y\alpha$  from the same $s$-secant $\alpha F'$ of $H'$ in $K'$. Let $\lambda_1,\lambda_2 \in V_1$. Since $ Tr_{F/F'}(\lambda_2 x)=x,  Tr_{F/F'}(\lambda_1 y)=y$,  it follows that
			$
			\lambda_2 x \ne \lambda_1 y, 
			$
			and so
			\[
			\dfrac{x\alpha}{\lambda_1} \ne \dfrac{y\alpha}{\lambda_2}. 
			\]
			This implies the line  $\alpha F$ in $K$ contains  $t=sq/r$ distinct points of $H$.  Hence, $H$ is a star-set. 
			\item 	Let $d =k+(q-1)i \in D'$. For each $\lambda \in F$, we have
			\[
			1/\lambda^d = \lambda^{q-1-d} = \lambda^{q-1-k},  
			\]
			so that   
			\[
			\pi_d(H)=\pi_d(H')\sum_{1/\lambda \in V_1}\lambda^d= \pi_{q-1-k}(V_1)\pi_d(H'). 
			\]
			From Lemma \ref{genTset},  if $\pi_{q-1-k}(V_1)\ne 0$, then
			\[
			k= \sum_{i=0}^{h-1} 2^ia_i,
			\]
			where $a_i \in \mathcal{A}$ for each $0 \le i \le h-1$. 
			By Lemma \ref{Usetd},  $\pi_d(H')=0$. 
This implies $\pi_d(H)=0$, and by Theorem \ref{Eset}, $H$ is a KM-arc of type $t$ with $t$-nucleus at $0$.  
\end{enumerate} 
\end{proof} 

Now we show that our construction is equivalent to the construction of   G\'acs and Weiner \cite{gacs2003}. 
Recall that they constructed a family of KM-arcs as follows.
Let $I$ be a direct complement of $F'$ in the additive group of $F$.  
Let $H_0$ be a KM-arc of type $s \ge 1$ with affine part $\{ (x_k:y_k:1) \mid x_k,y_k \in F'\}$. Construct the following point set
\[
J:=\{ (x_k:y_k+i:1) \mid (x_k:y_k:1) \in H_0, i \in I\}.
\]
Then $J$ can be uniquely extended to a KM-arc of type $t=sq/r$ in $PG(2,q)$.

Now we consider a collineation $\gamma$ defined by
\[
\gamma((a:b:c)) = (a:c:a+b+cd),
\]
where $d \in F'$ is chosen such that the line $x+y+zd=0$ does not intersect $H_0$, so $x_k+y_k+d \ne 0$. 
Under $\gamma$, the point $(x_k: y_k:1)$ of $H_0$ is mapped to
$\left(\frac{x_k}{x_k+y_k+d}:\frac{1}{x_k+y_k+d}:1 \right)
$, which corresponds to the element 
\[
p_k= \frac{x_k}{x_k+y_k+d}+\frac{1}{x_k+y_k+d}\mathbf{i} \in K.
\]
Also, the point $(x_k:y_k+i:1)$ of $J$ is mapped to
\[
\frac{x_k}{x_k+y_k+d+i}+\frac{1}{x_k+y_k+d+i}\mathbf{i}= \frac{p_k}{1+i(x_k+y_k+d)^{-1}}\in  \frac{p_k}{1+I}. 
\]
Therefore, 
$$\gamma (J) =  \left\{ \frac{p_k}{1+I} \right\}.$$ 

On the other hand,  if we consider $F$ as a vector space over $F'$ then $I$ is a subspace of codimension 1,  
not containing $F'$. 
Hence, $1+I$ is a coset, such that $(1+I) \cap F' = \{ 1\}$. Further, we can choose $b\in K^*$ such that 
$bV_1$ is parallel to $I$ (note that $V_0:= \{ x \in F \mid Tr_{F/F'}(x)=0 \}$ is a subspace in $F$ of codimension 1, and 
$bV_0$, $b \in K^*$, runs the set of all subspaces of codimension 1 in $F$). 
Let $bV_1 \cap F' = a$. Then $a^{-1}bV_1 \cap F' =1$ and $cV_1 \cap F' =1$ for 
$c=a^{-1}b$. Hence $cV_1 = 1+I$.  It shows the equivalence of our construction to the construction of G\'acs and Weiner.


\section{A family of KM-arcs type $t=q/r$} 

We present now some other examples of KM-arcs. 

\subsection{The construction}
Let $r=2^h$, where $h \mid m$. Let $F'=\mathbb{F}_{r}$, and $K'=\mathbb{F}_{r^2}$. We denote the group of $(r+1)$-st roots of unity in $K'$ by $S'$.  Let $\mathbf{u}$ be a generator of $S'$ whose conjugate over $F'$ is $\tilde{\mathbf{u}}=\mathbf{u}^{r}$. Let $b=\mathbf{u}+ \tilde{\mathbf{u}}$ be the trace of $\mathbf{u}$ over $F'$  and $(b_n)$ be the sequence in $F'$ defined by the recurrence relation
	$
	b_{n+1}=bb_n+b_{n-1},
$
 with initial terms $b_0=1,b_1=0$. Explicitly, for each $n$, 
 \[
 b_n = \dfrac{\mathbf{u}^{n-1}+\tilde{\mathbf{u}}^{n-1}}{\mathbf{u}+\tilde{\mathbf{u}}}.
 \]
              
Let $U:= \{u_i \mid 0 \le i \le r \}$, where $u_i= b_i+b_{i+1}\mathbf{i}$, for each $i$. 
Alternatively, if we identify each point $z=x+y\mathbf{i}$ as $[x \text{ } y]^T$,  then elements of $U$ are obtained recursively as
\[
u_{i+1}= \begin{bmatrix} 0 & 1 \\ 1 & b \end{bmatrix}u_i,
\]
with the initial term $u_0=1=[1 \text{ } 0]^T$.  We  also recall the relative trace map
$$
Tr_{F/\mathbb{F}_r}(x)= x^{q/r}+x^{q/r^2}+ \cdots +x^r+x.
$$
Define $V_1:= \{ x \in F \mid Tr_{F/\mathbb{F}_r}(x)=1 \}$ and $V_c:=cV_1$ for $c \in K^*$. 
In this subsection we prove the following theorem.
	
\begin{theorem} \label{Uconstruction} Let $c \in K^*$. The  set $H_r:= \{\lambda u \mid 1/\lambda \in V_c,u \in U\}$ is a KM-arc of type $t=q/r$ with $t$-nucleus at $0$ in $K$. 
	\end{theorem}

The proof of Theorem \ref{Uconstruction} is presented at the end of the subsection. 

\begin{lemma} \label{seqperiod}  
The set $U$ has size $r+1$. 
\end{lemma}
	\begin{proof} It is sufficient to prove that the sequence $(b_n)$ has period $r+1$. We have
\[
b_{r+1} = \dfrac{\mathbf{u}^{r}+\tilde{\mathbf{u}}^{r}}{\mathbf{u}+\tilde{\mathbf{u}}}=1, 
\]
\[
b_{r+2}= \dfrac{\mathbf{u}^{r+1}+\tilde{\mathbf{u}}^{r+1}}{\mathbf{u}+\tilde{\mathbf{u}}}=0, 
\]
so that the period of $(b_n)$ divides $r+1$. On the other hand,
$b_n=0$ if and only if $\mathbf{u}^{n-1}=\tilde{\mathbf{u}}^{n-1}$ if and only if $\mathbf{u}^{2(n-1)}=1$. This occurs if and only if $r+1$ divides $n-1$, which implies $b_{r+2}$ is the smallest zero term of the sequence excluding $b_1$. Therefore  $(b_n)$ has period $r+1$ and consequently $|U|=r+1$.  
	\end{proof}

Let $U':= \{u_i \mid 0 \le i \le r \}$, where $u_i= b_i+b_{i+1}\mathbf{u}$, for each $i$. It can be checked that
elements of $U'$ are pairwise distinct and of order $r+1$, so that $U'=S'$. This implies the following.
\begin{lemma} \label{Usetcoord} The set $U$ can be written as 
	\[
	U= \{x+y\mathbf{i} \mid x,y \in F', x^2+y^2+bxy=1 \}. 
	\]
\end{lemma} 
 We also recall from Section 4 the set
 \[
 \mathcal{A}= \left\{ \dfrac{q}{r^j} \mid 1 \le j \le m/h \right\} \cup \{0\}. 
 \]
 \begin{lemma} \label{genUset} Let $k=\sum_{j=0}^{h-1}2^ja_j \in \{1, \cdots,q-2\},$ where $a_j \in \mathcal{A}$.  Then for each $v \in S$,
 	\[
 	\sum_{u \in U} \langle v,u \rangle^k = 0.
 	\]
 \end{lemma} 
  	
\begin{proof}  \begin{enumerate}
 			\item    Let
 			\[
 			\mathcal{B} :=\left\{\sum_{j=0}^{h-1}2^jx_j \mid x_j \in \{0,a_j\} \right\}.
 			\]
 			We claim that for each $x \in \mathcal{B}$, 
 			\[
 			\sum_{i=0}^{r} b_i^xb_{i+1}^{k-x}=0.
 			\]  
 			Fix $x \in \mathcal{B}$. Let
 			\[y:=k-x=\sum_{j=0}^{h-1}2^jy_j,
 			\]
 			where $y_j:=a_j-x_j$ for each  $0 \le j \le h-1$.
 			We also introduce the sets
 			\[
 			\mathcal{B}_x := \left\{  \sum_{j=0}^{h-1} \pm 2^j x_j   \right\}, \text{ and } \mathcal{B}_y := \left\{  \sum_{j=0}^{h-1} \pm 2^j y_j   \right\}.
 			\]
 			For each $0\le i \le r$, we have
 			\begin{align*}
 			b_i^x &= 		\dfrac{(\mathbf{u}^{i-1}+\tilde{\mathbf{u}}^{i-1})^x}{(\mathbf{u}+\tilde{\mathbf{u}})^{x}} = 
 			\dfrac{\prod\limits_{j=0}^{h-1} (\mathbf{u}^{(i-1)2^jx_j}+\tilde{\mathbf{u}}^{(i-1)2^jx_j})}{(\mathbf{u}+\tilde{\mathbf{u}})^{x}}
 			=\dfrac{\sum\limits_{x' \in \mathcal{B}_x}\mathbf{u}^{(i-1)x'}}{(\mathbf{u}+\tilde{\mathbf{u}})^{x}}.
 			\end{align*}
 			Similarly, 
 			\[
 			b_{i+1}^{y}
 			= \dfrac{(\mathbf{u}^{i}+\tilde{\mathbf{u}}^{i})^{y}}{(\mathbf{u}+\tilde{\mathbf{u}})^{y}}
 			= \dfrac{\prod\limits_{j=0}^{h-1}(\mathbf{u}^{i2^jy_j}+\tilde{\mathbf{u}}^{i2^jy_j})}{(\mathbf{u}+\tilde{\mathbf{u}})^{y}}
 			=\dfrac{\sum\limits_{y' \in \mathcal{B}_y}\mathbf{u}^{iy'}}{(\mathbf{u}+\tilde{\mathbf{u}})^{y}}.
 			\]
 			Then	
 			\begin{align*}
 			\sum_{i=0}^{r}	b_i^xb_{i+1}^y =0 \iff
 			&\sum_{i=0}^{r} \left(\sum_{x' \in \mathcal{B}_x}\mathbf{u}^{(i-1)x'}\right)	\left(\sum_{y' \in \mathcal{B}_y}\mathbf{u}^{iy'}\right) =0\\
 			\iff	&\sum_{i=0}^{r} \sum_{\substack{x' \in \mathcal{B}_x \\y' \in \mathcal{B}_y}}\mathbf{u}^{i(x'+y')-x'} 
 			=\sum_{\substack{x' \in \mathcal{B}_x \\y' \in \mathcal{B}_y}}\sum_{i=0}^{r} \mathbf{u}^{i(x'+y')-x'}=0.
 			\end{align*}

 			%
 			%
 			%
 			%
 			%
 			%
 			%
 			We now consider the sum
 			\[
 			x'+y' =  \sum_{j=0}^{h-1}  2^j \left(\pm x_j\pm y_j \right).
 			\]
 			For each $j$, we have $x_j+y_j=a_j$ and 
 			$
 			x_j-y_j= 2x_j-a_j \in \{a_j,-a_j\}, 
 			$
 			so that
 			\[\pm x_j\pm y_j \in \{a_j,-a_j\}.\] 
 			We note that  $a_j \pmod{r+1} \in \{-1,0,1\}$. Also, $x'+y'=0$ implies $a_j=0$ for any $j$, which is impossible. 
 			It follows that  $x'+y' \ne 0 \pmod{r+1}. $ 
Therefore $\mathbf{u}^{x'+y'}$ generates a subgroup of $S'$ of order $\frac{r+1}{\gcd(x'+y',r+1)} \ne 1$. Consequently, 
 			\[
 			\sum_{i=0}^{r} \mathbf{u}^{ i(x'+y')-x'}=\mathbf{u}^{-x'}\sum_{i=0}^{r} \mathbf{u}^{ i(x'+y')}=0.
 			\] 
 			This proves $ \sum\limits_{i=0}^{r}	b_i^xb_{i+1}^{k-x}=0$. 
 			
 			\item To simplify the notation, let $A=\langle v,1 \rangle, B=\langle v, \mathbf{i} \rangle$. For fixed $0 \le i \le r$, we have
 			\begin{align*}
 			\langle v,u_i \rangle ^k &=\langle v, b_i+b_{i+1}\mathbf{i} \rangle ^k = \left(b_iA+b_{i+1}B \right)^{\sum_j2^ja_j} \\
 			&= \prod_{j=0}^{h-1} (b_i^{2^ja_j}A^{2^ja_j}+b_{i+1}^{2^ja_j}B^{2^ja_j}) = \sum_{x \in \mathcal{B}} b_i^xb_{i+1}^{k-x}A^xB^{k-x}.
 			\end{align*} 
 			From part 1),
 			\[	
 			\sum_{u \in U}	\langle v,u \rangle ^k = \sum_{i=0}^{r} \sum_{x \in \mathcal{B}} b_i^xb_{i+1}^{k-x}A^xB^{k-x} = \sum_{x \in \mathcal{B}} \left(\sum_{i=0}^{r} b_i^xb_{i+1}^{k-x}\right)A^xB^{k-x} =0.  
 			\] 
 		\end{enumerate}
 	\end{proof}

		\begin{proof}[Proof of Theorem \ref{Uconstruction}]				
			 We will assume that $c=1$ since multiplication by $c \in K^*$ is a collineation of $PG(2,q)$. 
			 	\begin{enumerate}
				\item  Let $x=b_i+b_{i+1}\mathbf{i},y=b_j+b_{j+1}\mathbf{i} \in U$, with $0 \le i,j\le r-1$, $i \ne j$.  
				Suppose that the lines $xF$ and $yF$ are the same. Since the matrix $ \begin{bmatrix}
				0 &1 \\ 1 & b
				\end{bmatrix}$
				permutes elements of $U$ cyclically, we may assume that $x=1$. Then $b_{j+1}=0$ and $y=1$, a contradiction. This implies for $x,y \in U$, the lines $xF$ and $yF$ intersect at $0$ only. Hence, $H_r$ is a star-set. 
					\item 	Let $d =k+(q-1)i \in D'$. We have
				\[
				\pi_d(H)=\pi_d(U)\sum_{1/\lambda \in V_1}\lambda^d= \pi_{q-1-k}(V_1)\pi_d(U). 
				\]
				From Lemma \ref{genTset},  if $\pi_{q-1-k}(V_1)\ne 0$, then
				\[
				k= \sum_{i=0}^{h-1} 2^ia_i,
				\]
				where $a_i \in \mathcal{A}$ for each $0 \le i \le h-1$. 
				From Lemma \ref{genUset}
				and by the method used in the proof of Theorem \ref{Eset}, we have $\pi_d(U)=0$.
				This implies $\pi_d(H_r)=0$, and by Theorem \ref{Eset}, $H_r$ is a KM-arc of type $t=q/r$ with $t$-nucleus at $0$.  
%
			\end{enumerate} 
			\end{proof}

\subsection{Automorphisms of $H_r$}
In this section, we  discuss some properties of the automorphism group of KM-arcs $H_r$ and show that they are  translation KM-arcs.  

An \textit{elation} with axis the line $l$ and centre the point $R$
on $l$ is a collineation which fixes the points of $l$ and stabilises the lines through the centre $R$. The set of all elations with a fixed centre and a fixed axis form a subgroup of $P\Gamma L(3,q)$. 

Let $H$ be a KM-arc of type $t>2$ with $t$-nucleus at $N$. Then $H$ is an \textit{elation KM-arc} with elation line $l_0$ if  for every $t$-secant $l \ne l_0$ to $H$, the group of elations with axis $l_0$ that stabilise $H$ acts transitively on the points of $l \cap H$.  
Furthermore, $H$ is a \textit{translation} KM-arc with translation line $l_0$ if the group of all elations with axis $l_0$ that stabilise $H$ acts transitively on the points of $H \backslash l_0$. In particular, every translation KM-arc is an elation KM-arc.

For $x,y \in F$, let $z=x+y\mathbf{i} \in K$.  
Let $l_0$ be the line $\langle 1,z \rangle =0$.  
We consider several maps from the automorphism group $P \Gamma L(3,q)$ of $PG(2,q)$. 
\begin{enumerate}

\item We define $\theta$ as a map induced by the matrix
		\[ 
	\begin{bmatrix}
	0 & 1 & 0  \\
	1 & b & 0  \\
	0 & 0 & 1 
	\end{bmatrix},
	\] 
which generates the cyclic group $C_{r+1}$.

\item 	The map $\sigma'$ is defined by
	\[
	\sigma':= \begin{bmatrix}
	1 & \delta+1 \\
	0 & b
	\end{bmatrix} \sigma,
	\]
	where $\sigma: z \mapsto z^2$ is a generator of $Gal(K/\mathbb{F}_2)$.  Explicitly, 
	\[
	\sigma'(z)=\sigma'(x+y\mathbf{i}) = x^2 + y^2+by^2 \mathbf{i}.
	\]
Note that $(\sigma')^m \in GL(2,q)$.  

\item Let $E_{a,b}$ be the map induced by the matrix 
\[ 
\begin{bmatrix}
1 & 0 & 0  \\
0 & 1 & 0  \\
a & b & 1 
\end{bmatrix},
\]
where $a,b \in F$. Explicitly,
\[
E_{a,b}(z) =\dfrac{z}{ax+by+1} = \dfrac{z}{\langle z,b+a\mathbf{i} \rangle +1} .
\]
We note that $E_{a,b}$ is an elation with centre $0$ and axis $\langle b+a\mathbf{i},z   \rangle = 0$. 
%
The  group 
\[
\mathcal{E}=\{ E_{a,b} \mid Tr_{F/F'}(a)=Tr_{F/F'}(b)=0 \}
\]
is a group of the automorphisms of $H_r$ that fixes every $t$-secant and acts transitively on the points of 
each $t$-secant of $H_r$. Then the group $\{ E_{0,b} \mid Tr_{F/F'}(b)=0 \}$
fixes all points of $l_0$ and acts transitively on the points   $H_r \cap l$ on  any  $t$-secant $l \ne l_0$. This shows that $H_r$ is an elation KM-arc with elation line $l_0$. 
 	\item Let  $\psi$ be the map induced by the matrix
\[
\begin{bmatrix}
1 & b &0\\
0 & 1 &0\\
0 & 0 &1
\end{bmatrix}.
\] Explicitly, 
$
\psi(z):= x+by+y\mathbf{i}. 
$
We note that $\psi$ is an elation with centre at infinity and axis $l_0$. Also, $\psi$ has order $2$. 
\item Let $\rho_{\gamma, s,t}$ be the map  induced by the matrix
\[ 
\begin{bmatrix}
\gamma & 0 & 0  \\
0 & \gamma^{-1} & 0  \\
s & t & 1 
\end{bmatrix},
\]
where $\gamma \in F'$, $s,t \in F$, $Tr_{F/F'}(s)=\gamma+1, Tr_{F/F'}(t)=\gamma^{-1}+1$. We note that 
\[ 
\begin{bmatrix}
\gamma & 0 & 0  \\
0 & \gamma^{-1} & 0  \\
s & t & 1 
\end{bmatrix}^k =
\begin{bmatrix}
\gamma^k & 0 & 0  \\
0 & \gamma^{-k} & 0  \\
s_k & t_k & 1 
\end{bmatrix}
,
\]
where $s_k=(\gamma^{k-1}+\cdots 1)s,$ and $t_k=(\gamma^{-(k-1)}+\cdots 1)t$. And so, if $\gamma$ is a primitive element of $F'$, then  $\rho_{\gamma,s,t}$ has order $r-1$. 
\end{enumerate}

\begin{theorem} $H_r$ is a translation KM-arc with translation line $l_0$. 
\end{theorem}

\begin{proof} 
	\begin{enumerate} \item We first show that the map $\psi$ is an automorphism of $H_r$. 
	Since $\psi$ is linear on $F$, it is sufficient to show that $\psi$   fixes the set $U$. We recall from Lemma  \ref{Usetcoord} that 
	\[
	U= \{x+y\mathbf{i} \mid x,y \in F', x^2+y^2+bxy=1 \}. 
	\]
	Let $u=x+y\mathbf{i}\in U$ with $\psi(u)= x'+y'\mathbf{i} = x+by+y\mathbf{i}$. Then
	\begin{align*}
	x'^2+y'^2+bx'y'&= (x+by)^2+y^2+ b(x+by)y \\
	&=  x^2+ y^2+bxy=1,
	\end{align*}
	which shows that $\psi(u) \in U$, and the claim follows. 	The action of $\psi$ on $H_r$ can then be described by its action on $U$. Since $\psi$ has order $2$, it has one fixed point (which is $1$) and $r/2$ orbits of size $2$.

\item We now show that the map $\rho_{\gamma,s,t}$ is an automorphism of $H_r$. Let $p=\lambda u$ be a point on $H_r$, where $u=x+y\mathbf{i} \in U$, and $\lambda \in F$ with $1/\lambda \in V_1$. We have
\[
\rho_{\gamma,s,t}(\lambda u) = \dfrac{\lambda}{B}\gamma x +\dfrac{\lambda}{B}\gamma^{-1} y \mathbf{i},
\]
where $B=s\lambda x+ t \lambda y +1$. We rewrite 
\begin{align*}
\rho_{\gamma,s,t}(\lambda u) &=
\dfrac{\lambda}{AB}(A\gamma x) +\dfrac{\lambda}{AB}(A\gamma^{-1} y) \mathbf{i} \\
&=\dfrac{\lambda}{AB}(x'+y' \mathbf{i}),
\end{align*}
where $x'=A\gamma x, y'=A\gamma^{-1} y$  and $A$ satisfies
\[
\dfrac{1}{A^2}= \gamma^2 x^2+\gamma^{-2} y^2 +bxy. 
\]
With this choice of $A$, we have
\[
x'^2+y'^2 +bx'y'=1,
\]
so that $x'+y'\mathbf{i} \in U$. Note that $A \in F'$, and so 
\begin{align*}
Tr_{F/F'}\left(\dfrac{AB}{\lambda}\right) &=  A^2Tr_{F/F'}\left(\dfrac{B^2}{\lambda^2}\right)  \\
&=A^2Tr_{F/F'}\left(\dfrac{s^2\lambda^2 x^2+ t^2 \lambda^2 y^2 +1}{\lambda^2}\right)\\
&= A^2 Tr_{F/F'}\left(s^2x^2+t^2y^2+\dfrac{1}{\lambda^2}\right)\\
&= A^2\left(x^2Tr_{F/F'}(s^2)+y^2Tr_{F/F'}(t^2)+1\right)\\
&=A^2(x^2(\gamma^2+1)+y^2(\gamma^{-2}+1)+1)\\
&=A^2( \gamma^2 x^2+\gamma^{-2} y^2 +bxy)=1,
\end{align*} which shows that $ \rho_{\gamma,s,t}(p) \in H_r$.  	
\item We consider a map $\rho:=\rho_{\gamma,s,t}$, where $\gamma$ is a primitive element of $F'$. In the following, we describe how $\rho$ permutes the $t$-secants of $H_r$. Using the calculations above, we see that for $u=x+y\mathbf{i} \in U$, the $t$-secant $uF$ is mapped to  $u'F$, where
$
u'=x'+y'\mathbf{i}= A\gamma x +A\gamma^{-1}y \mathbf{i}. 
$
It follows that $\rho^k$ maps $uF$ to $vF$, where
\[
v= A^k\gamma^kx+A^k\gamma^{-k}y \mathbf{i}.
\]
This implies that $\rho^k$ fixes the $t$-secant $uF$ if and only if $u\in \{1,i\}$ or $A^k\gamma^k=A^k\gamma^{-k}=1$. 
The latter case implies that $\gamma^{2k}=1$, which occurs if and only if $k=r-1$. 
Hence, $\rho$ fixes the $t$-secants $l_0$ and $\mathbf{i}F$, and $\langle \rho \rangle$ is transitive on the set of 
$(r-1)$ remaining $t$-secants of $H_r$.

\item Let $l$ be a $t$-secant of $H_r$ different from $\mathbf{i}F$ and $l_0$.	For a suitable $k\in \{ 1, \dots ,r-1 \}$, 
the map $\rho^k \psi \rho^{-k}$ is an elation that fixes $l_0$ and maps the line $\mathbf{i}F$ to $l$. 
Since $\rho^k \psi \rho^{-k}$ has order $2$,  it maps $l$ to the line $\mathbf{i}F$.
It follows that the group of automorphisms 

\[
\left\langle \rho^k \psi \rho^{-k} \mid k \in \{ 1, \dots ,r-1\} \right\rangle
\]

 is a group of elations with axis $l_0$ which is transitive on the set of $t$-secants $H\backslash l_0$. 
 Since $H_r$ is an elation KM-arc with elation line $l_0$, it follows that $H_r$ is a translation $KM$-arc with translation line $l_0$.  
\end{enumerate}
\end{proof}

We do not know if our general construction of KM-arcs $H_r$ is different from the one constructed by Korchmaros-Mazzocca, 
although their parameters look the same. 
Anyway, we see that some of the known KM-arcs can be easily described by polar coordinates.

\subsection{Examples}
We describe some examples of KM-arcs from the construction in Subsection 5.1. As before, let $F=\mathbb{F}_q$ and $K=\mathbb{F}_{q^2}$. We will also make use of the maps introduced in Subsection 5.2 and describe the full automorphism group of these examples. 
\begin{example}[KM-arcs of type $q/2$] Let $r=2$, $V_1= \{ x \in F \mid Tr_{F/\mathbb{F}_2}(x)=1 \}$, and $U=\{1, \mathbf{i}, \mathbf{i}+1\}$. Then $H_2:= \{\lambda u \mid 1/\lambda \in V_1,u \in U\}$ is a KM-arc of type $q/2$.   According to \cite[Corollary 4.4]{vandendriessche2011}, all KM-arcs of type $q/2$ are projectively equivalent and they are also known as \textit{projective triads}.  
	The automorphism group of $H_2$ is  
	$$G_2 = (\mathcal{E}. \langle \theta, \tau \rangle ).  \langle \sigma' \rangle \cong 
	(\mathcal{E}. S_3 ).  \langle \sigma' \rangle,$$
	where $\tau(x)=\bar{x}$ is the conjugation in $K$ and $S_3$ is the symmetric group. 
\end{example}
\begin{example}[KM-arcs of type $q/4$] Let $r=4$ and 
	$V_1= \{ x \in F \mid Tr_{F/\mathbb{F}_4}(x)=1 \}$.
	Let $\omega \ne 1$ be an element in $F$ such that $\omega^3=1$. 
	Define  $$U:=\{1, \mathbf{i}, 1+\omega \mathbf{i}, \omega+\mathbf{i},\omega + \omega  \mathbf{i}\}.$$
	Then  $H_4:= \{\lambda u \mid 1/\lambda \in V_1,u \in U\}$ is a KM-arc of type $q/4$. In \cite{geertrui2019}, it was proved that elation KM-arcs of type $q/4$ are translation KM-arcs  and the full classification was provided.
	By \cite[Theorem 3.12]{geertrui2019}, $H_4$ 
	is projectively equivalent to the one with parameters described therein. 
	The automorphism group of $H_4$ is 
	$$G_4= (\mathcal{E} . \langle \theta, \tau \rangle). \langle \sigma' \rangle \cong
	(\mathcal{E} . A_5). \langle \sigma' \rangle,$$
	where the map $\tau$ induced by the matrix
	\[
	\begin{bmatrix}
	1 & 1 &0\\
	0 & 1 &0\\
	0 & t &1
	\end{bmatrix},
	\] 
	with $t \in F$ such that $Tr_{F/F'}(t)=\omega$, and $A_5$ is the alternating group.   
\end{example}

\begin{example}[KM-arcs of type $q/8$] Let $r=8$ 
	and $V_1= \{ x \in F \mid Tr_{F/\mathbb{F}_8}(x)=1 \}$.
	Let $\eta$ be an element of $\mathbb{F}_{8}$ satisfying $\eta^3+\eta+1=0$.  Define 
	\[U:=\{
	1, \mathbf{i}, 
	1+\eta \mathbf{i}, 
	\eta + \eta^6  \mathbf{i},
	\eta^6 + \eta^3  \mathbf{i}, 
	\eta^3 + \eta^3  \mathbf{i},
	\eta^3 + \eta^6  \mathbf{i},
	\eta^6 + \eta  \mathbf{i},
	\eta+\mathbf{i}\}.
	\]   
	Then  $H_8:= \{\lambda u \mid 1/\lambda \in V_1,u \in U\}$ is a KM-arc of type $q/8$.
 Using GAP,  we calculated that the automorphism group of $H_8$ is 
 $$G_8= (\mathcal{E} . \langle \theta, \tau \rangle). \langle \sigma' \rangle,$$
 where the map $\tau$ is induced by the matrix
	\[
	 \begin{bmatrix}
		1 & 1 &0\\
		0 & 1 &0\\
		0 & t &1
	\end{bmatrix},
	\] 
with $t \in F$ such that $Tr_{F/F'}(t)=\eta^5$, and the group 
$(\mathcal{E} . \langle \theta, \tau \rangle)/ \mathcal{E} $ has order 504. 
Note that the automorphism $\tau$ is the conjugation in $K$ corrected by some element  $E_{0,b}$.

\end{example}


\bigskip

{\bf Acknowledgments}

\medskip

This research was supported by UAEU grant 31S366.  

%
%



\end{document}